\documentclass[12pt]{article}
\usepackage[margin=1in]{geometry}
\usepackage{amsthm, amsmath, amsfonts, amssymb}
\usepackage{graphicx}
\usepackage{authblk}

\newcommand{\N}[0]{\mathbb{N}}
\newcommand{\Z}[0]{\mathbb{Z}}

\newcommand{\F}[0]{\mathbb{F}}

\DeclareMathOperator{\constantterm}{ct}
\newcommand{\ct}[1]{\constantterm\left[#1\right]}
\newcommand{\ind}[0]{^{a, b}}

\newtheorem{claim}{Claim}
\newtheorem{conj}{Conjecture}

\newtheorem{thm}[claim]{Theorem}
\newtheorem{lemma}[claim]{Lemma}
\newtheorem{prop}[claim]{Proposition}
\newtheorem{cor}[claim]{Corollary}
\theoremstyle{definition}
\newtheorem{definition}{Definition}

\theoremstyle{remark}
\newtheorem*{rem}{Remark}

\theoremstyle{definition}

\title{Density and Symmetry in the Generalized Motzkin Numbers mod $p$}
\author{Nadav Kohen}
\affil{Indiana University\\Bloomington, IN, USA\\nkohen@iu.edu}
\date{January 8, 2025}

\begin{document}
\maketitle

\begin{abstract}
We give a formula for the density of $0$ in the sequence of generalized Motzkin numbers, $M\ind_n$, modulo a prime, $p$, in terms of the first $p$ generalized central trinomial coefficients $T\ind_n\bmod p$ (with $n<p$). We apply our method to various other sequences to obtain similar formulas. We also prove that $T\ind_{p-1-n}\equiv (b^2-4a^2)^{\frac{p-1}{2}-n}T\ind_n\pmod p$ to obtain tight lower bounds for the density of $0$ in our sequences. This symmetry of the first $p$ central trinomial coefficients mod $p$ also appears in a couple of other applications, including the proof of a novel symmetry of the first $p-2$ Motzkin numbers that is of independent interest: $M\ind_{p-3-n}\equiv (b^2-4a^2)^{\frac{p-3}{2}-n}M\ind_n\pmod p$.
\end{abstract}

\section{Introduction}
The Motzkin numbers (A001006 of \cite{oeis}), $M_n$, count the number of lattice paths from the origin to $(n, 0)$, which do not go below the x-axis, with steps UP $= (1, 1)$, LEVEL $= (1, 0)$, and DOWN $= (1, -1)$. See \cite{motzkinsurvey} for many other combinatorial settings in which the Motzkin numbers arise. Similarly, the \emph{generalized} Motzkin numbers, $M\ind_n$, count the same lattice paths but where there are $a$ distinct colors for UP and DOWN steps and $b$ distinct colors for LEVEL steps, and two paths are only equal if they have the same steps and colors at each step. All of the results for the Motzkin numbers in this paper work in the generality of the generalized Motzkin numbers.\\

Some work has been done to characterize $M_n$ and similar sequences modulo various prime powers. For example, Deutsch and Sagan \cite{deutschsagan} characterized $M_n\bmod 3$. They also described all $n$ such that $M_n\equiv 0 \pmod p$ for $p=2,4,5$.\\

In Proposition \ref{MainResult} and Corollary \ref{MainCorollary} of this paper, we partially answer the question, ``What is the density of the (generalized) Motzkin numbers that are divisible by a prime $p$?'' Using these results, we reprove a lower bound of $\frac{2}{p(p-1)}$ for the density of $0$ in $M_n\bmod p$ (for $p>2$), which has previously been drawn as a conclusion of Theorem 5 in \cite{burnspaper} by Burns. Along the way, we also prove the novel symmetry $M\ind_{p-3-n}\equiv (b^2-4a^2)^{\frac{p-3}{2}-n}M\ind_n$ (Theorem \ref{MSym}), and in Corollary \ref{Mdiv} we provide a characterization of $n$ such that $M\ind_n\equiv 0\pmod p$ from which existing results such as that of Deutsch and Sagan for $p=5$ can be recovered. All of this is accomplished by reducing the study of the (generalized) Motzkin numbers modulo $p$ to the study of the lesser-known (generalized) central trinomial coefficients modulo $p$, which have a simpler characterization provided by Proposition \ref{lucas}. In the final section of this paper, we demonstrate the general applicability of our approach by applying it to a few other sequences from the online encyclopedia \cite{oeis}. This analysis yields a novel result (Corollary \ref{same_density}): that for every prime, $p$, the density of $0$ in the Riordan numbers (A005043 of \cite{oeis}) mod $p$ is equal to the density of $0$ in the sequence A005773 of \cite{oeis} mod $p$.\\

The central trinomial coefficients (A002426 of \cite{oeis}), $T_n$, count the number of lattice paths from the origin to $(n, 0)$ using the same steps (UP, LEVEL, and DOWN) as Motzkin paths, but where there is no restriction that the paths not go below the x-axis (these are sometimes called \emph{grand} Motzkin paths). As with the generalized Motzkin numbers, the \emph{generalized} central trinomial coefficients, $T\ind_n$ have $a$ colors associated with UP and DOWN steps and $b$ colors associated with LEVEL steps. The central trinomial coefficients are so named because $T\ind_n = \ct{(ax^{-1} + b + ax)^n}$ where $ct$ extracts the ``constant term'' of a Laurent polynomial. Under this paradigm, we note that $M\ind_n = \ct{(ax^{-1} + b + ax)^n\cdot(1-x^2)}$ because the paths to $(n, 0)$ that do cross the x-axis can be bijected with arbitrary paths to $(n,2)$.\\

For readers familiar with the Rowland-Zeilberger automaton \cite{rowlandzeilberger}, the intuition that the Motzkin numbers modulo $p$ should be studied via the central trinomial coefficients comes from the observation that a random walk on the finite state machine for the Motzkin numbers quickly lands within, and never leaves, a sub-graph corresponding to the finite state machine for the central trinomial coefficients, which have a simple algebraic description given by Proposition \ref{lucas}. The formal instantiation of this intuition, by rewriting $M\ind_n\bmod p$ using combinations of central trinomial coefficients, is provided in Propositions \ref{MtoT} and \ref{M3case}. This approach of studying $M\ind_n\bmod p$ via the central trinomial coefficients has previously been used by the author in \cite{LastPaper} to show that $M\ind_n\bmod p$ is uniformly recurrent if and only if $p\nmid T\ind_n$ for all $n < p$.\\

The author would like to thank his advisor, Professor Michael Larsen, for his support and for observing and originally proving the symmetry of Theorem \ref{MSym} by other means, which began the author's path in this research direction.

\subsection{Notation and Conventions}
Throughout this paper, the superscripts $a$ and $b$ are omitted from $M\ind_n$ and $T\ind_n$ when $a=b=1$ (i.e., $M_n = M^{1,1}_n$ and $T_n = T^{1,1}_n$). If $Q(x)$ is a Laurent polynomial, then $\ct{Q(x)}$ denotes the constant term of $Q(x)$ (i.e., the coefficient of $x^0$) and $\deg Q(x)$ denotes the absolute value of the largest exponent of $x$ (positive or negative) appearing in $Q(x)$ with non-zero coefficient. If $\Sigma$ is a set, $\Sigma^*$ denotes the set of words (i.e., strings) of any length whose characters are from $\Sigma$ (including the empty word). If $n$ is a non-negative integer, and $p$ is a prime, then let $(n)_p\in\F_p^*$ be the word whose characters are the digits of $n$ in base $p$. That is, if we let $(n)_p[i]$ denote the $i$th digit in the base-$p$ expansion of $n$ so that $n = \sum_{i\in\Z_{\geq 0}} (n)_p[i]p^i$, then $(n)_p = ((n)_p\left[\text{length}(n_p)-1\right])\cdots((n)_p[1])((n)_p[0])$. Note that when working with strings, exponents denote repetition. For example, $(p-1)^k\in\F_p^*$ denotes a run of $k$ characters that are all the character $(p-1)$. Also note that every statement made in this paper about $(n)_p$ should also hold for $0^k(n)_p$ for every $k$.\\

In this paper, we prove results dealing with the density of certain values within sequences:

\begin{definition} The \emph{asymptotic density}, or just \emph{density}, of a subset, $S\subseteq \N$, is $$\lim_{N\rightarrow\infty}\frac{S\cap\{n\in \N\mid n<N\}}{N}.$$ When it exists, this is equal to $$\lim_{N\rightarrow\infty}\frac{S\cap\{n\in \N\mid n<p^N\}}{p^N},$$ which is the form we primarily use. Lastly, if $p$ is a prime, we say that the density of a value $x\in\F_p$ in a sequence $a_n$ over $\F_p$ is the usual density of $\{n\in\N\mid a_n = x\}$.
\end{definition}

\section{Symmetry}\label{SymSection}

In this section, we prove a novel result relating $T\ind_k$ and $T\ind_{p-1-k}$ that is useful in proving the primary results of this paper in the next section. We also derive the corresponding result for $M\ind_k$ and $M\ind_{p-3-k}$.\\

We follow the philosophy that the mod $p$ study of the generalized Motzkin numbers, $M\ind_n$, is reducible to the study of their corresponding generalized central trinomial coefficients, $T\ind_n$, which are well-structured and determined by their first $p$ elements (see Proposition \ref{lucas}). Thus, we begin by stating an explicit rule for this reduction, which is itself a specific instance of a family of such rules for reducing to the study of $T\ind_n$ for any sequence of the form $\ct{(ax^{-1} + b + ax)^nQ}$ where $Q$ is some Laurent polynomial in $x$ (see Section 4 of \cite{LastPaper} for a discussion of this family of equations).

\begin{prop}\label{MtoT}
$$2a^2M\ind_n = (4a^2 - b^2)T\ind_n + 2bT\ind_{n+1} - T\ind_{n+2},$$
and in particular,
$$2M_n = 3T_n + 2T_{n+1} - T_{n+2}.$$
\end{prop}
\begin{proof}
Let $P = ax^{-1} + b + ax$ so that $T\ind_n = \ct{P^n}$. Define $A_n := \ct{x\cdot P^n} = \ct{x^{-1}\cdot P^n}$ and $B_n := \ct{x^2\cdot P^n} = \ct{x^{-2}\cdot P^n}$ so that $M\ind_n = \ct{(1-x^2)\cdot P^n} = T\ind_n - B_n$.\\

From $T\ind_{n+1} = \ct{(ax^{-1} + b + ax)^{n+1}} = \ct{(ax^{-1} + b + ax)\cdot(ax^{-1} + b + ax)^n} = 2aA_n + bT\ind_n$ we have that $2aA_n = T\ind_{n+1} - bT\ind_n$. Similarly from $A_{n+1} = aT\ind_n + bA_n + aB_n$, we find that
\begin{align*}T\ind_{n+2} &= 2aA_{n+1} + bT\ind_{n+1}\\
&= 2a(aT\ind_n + bA_n + aB_n) + b(2aA_n + bT\ind_n)\\
&= (b^2 + 2a^2)T\ind_n + 2b(2aA_n) + 2a^2B_n\\
&= (b^2 + 2a^2)T\ind_n + 2b(T\ind_{n+1} - bT\ind_n) + 2a^2B_n
\end{align*}
and thus $2a^2B_n = T\ind_{n+2} - 2bT\ind_{n+1} + (b^2 - 2a^2)T\ind_n$.\\

In conclusion, $2a^2M\ind_n = 2a^2T\ind_n - 2a^2B_n = (4a^2 - b^2)T\ind_n + 2bT\ind_{n+1} - T\ind_{n+2}$.
\end{proof}

Next, we give a mostly combinatorial proof of a general two-term recurrence for the sequences $T\ind_n$, which is analogous to the well-known two-term recurrence for the Motzkin numbers (see Corollary \ref{Mrecur}).

\begin{prop}\label{Trecur}
$$nT\ind_n = b(2n-1)T\ind_{n-1} - (b^2 - 4a^2)(n-1)T\ind_{n-2},$$
and in particular,
$$nT_n = (2n-1)T_{n-1} + (3n-3)T_{n-2}.$$
\end{prop}
\begin{proof}
Recall that $M\ind_n$ counts the number of lattice paths from $(0,0)$ to $(n, 0)$ using $a$ distinct UP and DOWN steps, $(1, 1)$ and $(1, -1)$ respectively, and $b$ distinct LEVEL steps, $(1, 0)$, which do not go below the x-axis. Additionally, $T\ind_n$ simply counts the number of such paths without the x-axis restriction, and $A_n = \ct{x\cdot P^n}$ (as in Proposition \ref{MtoT}) does the same but counting paths to $(n, 1)$.\\

Let $\gamma_n$ count the number of paths from the origin to $(n,1)$ such that the only intersection with the x-axis is the starting point. Then $\gamma_n = aM\ind_{n-1}$ since adding an UP step to the beginning of a path (there are $a$ ways to do this) counted by the $M\ind_{n-1}$ gives a unique path counted by $\gamma_n$. Furthermore, $n\gamma_n = A_n$ since there are $n\gamma_n$ paths from the origin to $(n,1)$ that never return to the x-axis and which have a special vertex labeled; this labeled vertex partitions the path into a front and a back (consisting of UP, DOWN, and LEVEL steps) $P = FB$; Now the path $Q = BF$ is an arbitrary path to $(n, 1)$, counted by $A_n$, whose rightmost-lowest point is the labeled vertex (making the process reversible, and hence a bijection). Putting these together, we get $A_n = n\gamma_n = anM\ind_{n-1}$.\\

We now combine that $T\ind_n = 2aA_{n-1} + bT\ind_{n-1}$ (as discussed in the proof of Proposition \ref{MtoT}), $A_{n-1} = a(n-1)M\ind_{n-2}$, and Proposition \ref{MtoT} to get
\begin{align*}
T\ind_n &= 2aA_{n-1} + bT\ind_{n-1}\\
&= (n-1)(2a^2M\ind_{n-2}) + bT\ind_{n-1}\\
&= (n-1)\left((4a^2 - b^2)T\ind_{n-2} + 2bT\ind_{n-1} - T\ind_n\right) + bT\ind_{n-1}\\
&= -(n-1)T\ind_n + (2bn - 2b + b)T\ind_{n-1} - (b^2 - 4a^2)(n-1)T\ind_{n-2}.
\end{align*}
From which we can conclude $nT\ind_n = b(2n-1)T\ind_{n-1} - (b^2-4a^2)(n-1)T\ind_{n-2}$.
\end{proof}

In fact, the already-known analogous recurrence for the generalized Motzkin numbers can be derived directly from the previous two results by a tedious algebraic manipulation. An independent proof is also given by Woan in \cite{WeightedMotzkin}.

\begin{cor}\label{Mrecur}
$$(n+2)M\ind_n = b(2n+1)M\ind_{n-1} - (b^2-4a^2)(n-1)M\ind_{n-2},$$
and in particular,
$$(n+2)M_n = (2n+1)M_{n-1} + (3n-3)M_{n-2}.$$
\end{cor}\qed

As it turns out, Proposition \ref{Trecur} is sufficient to prove our symmetry directly!

\begin{thm}\label{Tsym}
$$T\ind_{p-1-k}\equiv (b^2-4a^2)^{\frac{p-1}{2} - k}T\ind_k\pmod p$$
for $p > 2$ and $0\leq k\leq\frac{p-1}{2}$, and in particular,
$$T_{p-1-k}\equiv (-3)^{\frac{p-1}{2} - k}T_k.$$
\end{thm}
\begin{proof}
By induction on $k\leq \frac{p-1}{2}$. For $k = \frac{p-1}{2}$, $T\ind_{p-1 - \frac{p-1}{2}} = T\ind_{\frac{p-1}{2}} = (b^2-4a^2)^0T\ind_{\frac{p-1}{2}}$. For $k=\frac{p-3}{2}$, we wish to show that $T\ind_{\frac{p+1}{2}}\equiv (b^2 - 4a^2)T\ind_{\frac{p-3}{2}}$, which follows from Proposition \ref{Trecur} since $2^{-1}T\ind_{\frac{p+1}{2}} \equiv \frac{p+1}{2}T\ind_{\frac{p+1}{2}} = bpT\ind_{\frac{p-1}{2}} - (b^2 - 4a^2)\left(\frac{p-1}{2}\right)T\ind_{\frac{p-3}{2}}\equiv 2^{-1}(b^2-4a^2)T\ind_{\frac{p-3}{2}}$.\\

Now let $0\leq k < \frac{p-3}{2}$. Then using Proposition \ref{Trecur} and induction,
\begin{align*}
(p-1-k)T\ind_{p-1-k} &= b(2(p-1-k)-1)T\ind_{p-2-k} - (b^2 - 4a^2)(p-2-k)T\ind_{p-3-k}\\
&\equiv -b(2k + 3)T\ind_{p-2-k} + (b^2 - 4a^2)(2+k)T\ind_{p-3-k}\\
&\equiv -b(2k + 3)(b^2-4a^2)^{\frac{p-3}{2} - k}T\ind_{k+1} + (2+k)(b^2 - 4a^2)^{\frac{p-3}{2} - k}T\ind_{k+2}\\
&= (b^2-4a^2)^{\frac{p-3}{2} - k}\left((k+2)T\ind_{k+2} - b(2k+3)T\ind_{k+1}\right)\\
&= (b^2-4a^2)^{\frac{p-3}{2} - k}\left(b(2k+3)T\ind_{k+1} - (b^2-4a^2)(k+1)T\ind_k - b(2k+3)T\ind_{k+1}\right)\\
&= (-k-1)(b^2-4a^2)^{\frac{p-1}{2}-k}T\ind_k\\
&\equiv (p-1-k)(b^2-4a^2)^{\frac{p-1}{2}-k}T\ind_k.
\end{align*}
\end{proof}

In the case that $p\nmid b^2 - 4a^2$, then $b^2-4a^2$ is invertible so that we can also conclude that $T\ind_k\equiv (b^2-4a^2)^{k - \frac{p-1}{2}}T\ind_{p-1-k}\pmod p$. Of course, in the case that $p\mid b^2-4a^2$ we instead conclude that $T\ind_k\equiv 0\pmod p$ for all $\frac{p-1}{2} < k < p$.\\

Before moving on, we quickly mention that an easy Corollary of Theorem \ref{Tsym} is that $T\ind_{p-1}\equiv (b^2-4a^2)^{\frac{p-1}{2}}$ so that, assuming $p\nmid b^2-4a^2$, $(T\ind_{p-1})^2\equiv 1$ by Fermat's little theorem.\\

\begin{rem}
Theorem \ref{Tsym} seems to help explain the observed density of the sequence of primes that do not divide any central trinomial coefficient (A113305 of OEIS \cite{oeis}); for primes less than $10^6$ the density of these primes is near $0.6075$. As a corollary of Proposition \ref{lucas} of the next section, $p$ is in this sequence if and only if it does not divide any of the first $p$ central trinomial coefficients. But the author was originally surprised by how high this density seems to be, since if we pretend that on average $T_n$ ($n<p$) has a $\frac{1}{p}$ chance of being congruent to $0$ mod $p$, then we expect the density to be $\lim_{p\rightarrow\infty}\left(1 - \frac{1}{p}\right)^p = e^{-1}\approx 0.3679$. But in view of Theorem \ref{Tsym}, we now know that $p$ does not divide any central trinomial coefficient if and only if it does not divide the first $\frac{p+1}{2}$ of them! Now if we make the same assumption we instead get the density estimate $\lim_{p\rightarrow\infty}\left(1 - \frac{1}{p}\right)^{\frac{p+1}{2}} = e^{-\frac{1}{2}}\approx 0.6065$, which is much closer to the experimental value.\\
\end{rem}

The analogous result to Theorem \ref{Tsym} for the generalized Motzkin numbers can be proven directly from their two-term recurrence in a very similar fashion to the proof of Theorem \ref{Tsym}. However, we instead provide a proof using Proposition \ref{MtoT} and Theorem \ref{Tsym}.

\begin{thm}\label{MSym}
$$M\ind_{p-3-k}\equiv (b^2-4a^2)^{\frac{p-3}{2} - k}M\ind_k\pmod p$$
for $p>3$ and $0\leq k\leq\frac{p-3}{2}$, and in particular,
$$M_{p-3-k}\equiv (-3)^{\frac{p-3}{2} - k}M_k.$$
\end{thm}
\begin{proof}
\begin{align*}
2a^2M\ind_{p-3-k} &= (4a^2 - b^2)T\ind_{p-3-k} + 2bT\ind_{p-2-k} - T\ind_{p-1-k}\\
&\equiv -(b^2-4a^2)^{\frac{p-3}{2} - k}T\ind_{k+2} + 2b(b^2-4a^2)^{\frac{p-3}{2}-k}T\ind_{k+1} - (b^2-4a^2)^{\frac{p-1}{2}-k}T\ind_k\\
&=(b^2-4a^2)^{\frac{p-3}{2}-k}\left(-T\ind_{k+2} + 2bT\ind_{k+1} + (4a^2-b^2)T\ind_k\right)\\
&= (b^2-4a^2)^{\frac{p-3}{2}-k}(2a^2M\ind_k).
\end{align*}
This gives us the theorem when $a$ is not a multiple of $p$, and otherwise $$M\ind_n = \ct{(ax^{-1} + b + ax)^n(1-x^2)}\equiv \ct{b^n(1-x^2)} = b^n\pmod p$$ in which case $(b^2-4a^2)^{\frac{p-3}{2}-k}M\ind_{k}\equiv b^{p-3-2k}b^k = b^{p-3-k}\equiv M\ind_{p-3-k}$, as desired.
\end{proof}

As an additional application of Theorem \ref{Tsym} we prove the following conjecture (Batalov 2022) from the OEIS page for the Motzkin numbers (A001006 of \cite{oeis}), which can also be inferred from Tables 4 and 5 of \cite{burnspaper}: If $p$ is a prime of the form $6m+1$, then $M_{p-2}$ is divisible by $p$. In fact, we can prove a stronger result:

\begin{prop} A prime $p$ divides $M_{p-2}$ if and only if $p\equiv 1\pmod 3$.
\end{prop}
\begin{proof}
For $p=2$, $2\nmid M_0$ and $2\not\equiv 1\pmod 3$, so assume $p>2$. Note that $T_p\equiv 1$ (see Proposition \ref{lucas}). We prove the proposition by reducing to central trinomial coefficients:
\begin{align*}
2M_{p-2} &= 3T_{p-2} + 2T_{p-1} - T_p\\
&\equiv 3(-3)^{\frac{p-3}{2}} + 2(-3)^{\frac{p-1}{2}} - 1\\
&= -(-3)^{\frac{p-1}{2}} + 2(-3)^{\frac{p-1}{2}} - 1\\
&= (-3)^{\frac{p-1}{2}} - 1
\end{align*}
and therefore, by quadratic reciprocity, $p\mid M_{p-2}\Leftrightarrow (-3)^{\frac{p-1}{2}}\equiv 1\Leftrightarrow p\equiv 1\pmod 3$.
\end{proof}

\section{Density}

In this section, we answer the question: ``What is the density of the Motzkin numbers that are divisible by $p$?'' In subsection \ref{other_seq}, we do the same for some additional sequences. In all cases, answering this question is accomplished by further elaborating the consequences Proposition \ref{MtoT} (or analogous results) and Theorem \ref{Tsym} alongside the following simple description of $T\ind_n\bmod p$.\\

\begin{prop}\label{lucas}
For every prime $p$, if $a_n = \ct{P^n}$ where $P$ is a Laurent polynomial with $\deg P = 1$, then $a_n\equiv \prod a_{(n)_p[i]}\pmod p$. In particular, $T\ind_n\equiv \prod T\ind_{(n)_p[i]}\pmod p$.
\end{prop}
\begin{proof}
We induct on the number of digits in $(n)_p$. Certainly if $n = (n)_p[0] < p$, then $a_n = a_{(n)_p[0]}$. Otherwise, if $n = qp + (n)_p[0]$, then
\begin{align*}
a_n &= \ct{P(x)^{qp + (n)_p[0]}}\\
&\equiv \ct{P(x^p)^qP(x)^{(n)_p[0]}}\pmod p &\left(P(x)^p\equiv P(x^p)\pmod p\right)\\
&= \ct{P(x^p)^q}\ct{P(x)^{(n)_p[0]}} &\left((n)_p[0]<p\text{ so there is no cancellation}\right)\\
&= \ct{P(x)^q}\ct{P(x)^{(n)_p[0]}} &\left(\ct{P(x^k)^n} = \ct{P(x)^n}\right)\\
&= a_qa_{(n)_p[0]}\\
&= \prod a_{(n)_p[i]} &(\text{by induction, since }(q)_p\text{ has fewer digits than }(n)_p).
\end{align*}
\end{proof}

Note that this proposition shows that the sequence $T\ind_n\bmod p$ can be described by its first $p$ terms, making Theorem \ref{Tsym} more important than it may initially seem.\\

Given the base-$p$ digit-multiplicative nature of $T\ind_n$, every combination of offsets $T\ind_{n+i}$ can be understood by ``factoring out'' common digits between the $(n+i)_p$ for various $i$. For example, if $(n)_p = qn_0$ and $n_0\neq p-1$, then $\alpha T\ind_n + \beta T\ind_{n+1} \equiv \alpha T\ind_qT\ind_{n_0} + \beta T\ind_qT\ind_{n_0 + 1} = T\ind_q\left(\alpha T\ind_{n_0} + \beta T\ind_{n_0 + 1}\right)$.\\

A version of this idea, sufficient for our purposes, is encapsulated in the following formalism where we bound the maximum distance between indices by $p$; this bound ensures that the following only has two summands. This is sufficient for our purposes since $h=2$ below for the Motzkin numbers.

\begin{lemma}\label{sumform}
If $a_n = \ct{P^n}$ where $\deg P = 1$, $b_n = \sum_{i=0}^h\alpha_ia_{n+i}$, $C_{n_0} = \min(h, p-n-1)$, $p > h$, and $(n)_p = qm(p-1)^kn_0$ with $q\in\F_p^*,m,n_0\in\F_p$ and $k\geq 0$, then $$b_n\equiv a_q\left(a_ma_{p-1}^k\sum_{i=0}^{C_{n_0}}\alpha_ia_{n_0+i} + a_{m+1}\sum_{i=p-n_0}^h\alpha_ia_{i-(p-n_0)}\right)\pmod p.$$ In particular if $n_0 < p-h\Leftrightarrow h<p-n_0$, then $$b_n\equiv a_q\left(a_ma_{p-1}^k\sum_{i=0}^h\alpha_ia_{n_0+i}\right) \equiv a_{\frac{n-n_0}{p}}b_{n_0}\pmod p.$$
\end{lemma}
\begin{proof}
If $n_0+i\leq p-1$, then $(n+i)_p = qm(p-1)^k(n_0+i)$ while if $n_0+i\geq p$, then $(n_0+i)_p = q(m+1)0^k(n_0+i-p)$. Thus, noting that $a_0=1$,
\begin{align*}
b_n &= \sum_{i=0}^h\alpha_ia_{n+i}\\
&= \sum_{i=0}^{C_{n_0}}\alpha_ia_{n+i} + \sum_{i=p-n_0}^h\alpha_ia_{n+i}\\
&\equiv \sum_{i=0}^{C_{n_0}}\alpha_i(a_qa_ma_{p-1}^ka_{n_0+i}) + \sum_{i=p-n_0}^h\alpha_i(a_qa_{m+1}a_{n_0+i-p})\\
&= a_q\left(a_ma_{p-1}^k\sum_{i=0}^{C_{n_0}}\alpha_ia_{n_0+i} + a_{m+1}\sum_{i=p-n_0}^h\alpha_ia_{n_0+i-p}\right)
\end{align*}
\end{proof}

We now apply this lemma's factorization to $2a^2M\ind_n$ via Proposition \ref{MtoT}. Additionally, we make use of Theorem \ref{Tsym} multiple times to produce a more elegant final form. Note that the third case in the formula for $2a^2M\ind_n$ below (where $n_0 = p-1$) is more elegant after our application of Theorem \ref{Tsym} because we intend to set $M\ind_n$ congruent to $0$ in what follows.

\begin{prop}\label{M3case} If $p\nmid b^2-4a^2$ and if we let $\ell = p-2-m$, then modulo $p>2$, $$2a^2M\ind_n\equiv\left\{\begin{array}{cc}
2a^2T\ind_qM\ind_{n_0} & (n)_p = qn_0, n_0 < p-2\\
T\ind_q\left(bT\ind_m(T\ind_{p-1})^{k+1} - T\ind_{m+1}\right) & (n)_p = qm(p-1)^k(p-2)\\
(b^2 - 4a^2)^{m+1 - \frac{p-1}{2}}T\ind_q\left(bT\ind_\ell - (T\ind_{p-1})^{k+1}T\ind_{\ell +1}\right) & (n)_p = qm(p-1)^k(p-1).
\end{array}\right.$$
In particular, when $a=b=1$, we get that $$2M_n\equiv\left\{\begin{array}{cc}
2T_qM_{n_0} & (n)_p = qn_0, n_0 < p-2\\
T_q\left(T_mT_{p-1}^{k+1} - T_{m+1}\right) & (n)_p = qm(p-1)^k(p-2)\\
(-3)^{m+1 - \frac{p-1}{2}}T_q\left(T_\ell - T_{p-1}^{k+1}T_{\ell +1}\right) & (n)_p = qm(p-1)^k(p-1)
\end{array}\right.\pmod p.$$
\end{prop}
\begin{proof}
Since $2a^2M\ind_n = (4a^2-b^2)T\ind_n + 2bT\ind_{n+1} - T\ind_{n+2}$ by Proposition \ref{MtoT}, Lemma \ref{sumform} tells us that if $(n)_p=qn_0$ with $n_0 < p-2$, then $2a^2M\ind_n\equiv T\ind_q\left(2a^2M\ind_{n_0}\right)$.\\

The lemma also tells us that if $(n)_p = qm(p-1)^k(p-2)$, then $$2a^2M\ind_n\equiv T\ind_q\left(T\ind_m(T\ind_{p-1})^k((4a^2-b^2)T\ind_{p-2} + 2bT\ind_{p-1}) + T\ind_{m+1}(-T\ind_0)\right).$$ In this case we can use Theorem \ref{Tsym} (and $T\ind_1 = b$) to see that $T\ind_{p-1}\equiv (b^2-4a^2)^{\frac{p-1}{2}}$ and $T\ind_{p-2}\equiv (b^2-4a^2)^{\frac{p-1}{2}-1}T\ind_1 = (b^2-4a^2)^{-1}T\ind_{p-1}b$ (noting that $p\nmid b^2-4a^2$), thus $(4a^2-b^2)T\ind_{p-2} + 2bT\ind_{p-1}\equiv bT\ind_{p-1}$. This yields our desired form $$2a^2M\ind_n\equiv T\ind_q\left(bT\ind_m(T\ind_{p-1})^{k+1} - T\ind_{m+1}\right).$$

Lastly, the lemma (alongside Theorem \ref{Tsym}) tells us that if $(n)_p = qm(p-1)^k(p-1)$, then
\begin{align*}
2a^2M\ind_n&\equiv T\ind_q\left(T\ind_m(T\ind_{p-1})^k((4a^2-b^2)T\ind_{p-1}) + T\ind_{m+1}(2bT\ind_{0} - T\ind_1)\right)\\
&= T\ind_q\left(-(b^2-4a^2)T\ind_m(T\ind_{p-1})^{k+1} + bT\ind_{m+1}\right)\\
&\equiv T\ind_q\left(-(b^2-4a^2)^{m+1-\frac{p-1}{2}}T\ind_{p-1-m}(T\ind_{p-1})^{k+1} + b(b^2-4a^2)^{m+1-\frac{p-1}{2}}T\ind_{p-2-m}\right)\\
&= (b^2-4a^2)^{m+1-\frac{p-1}{2}}T\ind_q\left(bT\ind_\ell - T\ind_{\ell+1}(T\ind_{p-1})^{k+1}\right).
\end{align*}
\end{proof}

We now set both sides of the above equality congruent to $0$ so that computing the density of Motzkin numbers divisible by $p$ becomes a simple matter of bookkeeping.

\begin{cor}\label{Mdiv} If $p\nmid b^2-4a^2$, $\ell = p-2-m$, $p > 2$, and if $T\ind_n\not\equiv 0$ for all $n<p$, then $$M\ind_n\equiv 0\Leftrightarrow\left\{\begin{array}{cc}
M\ind_{n_0} \equiv 0 & (n)_p = qn_0, n_0<p-2\\
bT\ind_m\equiv (T\ind_{p-1})^{k+1}T\ind_{m+1} & (n)_p = qm(p-1)^k(p-2)\\
bT\ind_\ell\equiv (T\ind_{p-1})^{k+1}T\ind_{\ell+1} & (n)_p = qm(p-1)^k(p-1).
\end{array}\right.$$ Note that $0\leq m < p-1\Rightarrow 0\leq \ell < p-1$ as well.\\
In particular, when $a=b=1$, then $$M_n\equiv 0\Leftrightarrow \left\{\begin{array}{cc}
M_{n_0} \equiv 0 & (n)_p = qn_0, n_0<p-2\\
T_m\equiv T_{p-1}^{k+1}T_{m+1} & (n)_p = qm(p-1)^k(p-2)\\
T_\ell\equiv T_{p-1}^{k+1}T_{\ell+1} & (n)_p = qm(p-1)^k(p-1).
\end{array}\right.$$
\end{cor}\qed

Corollary \ref{Mdiv} generalizes existing results about the divisibility of the Motzkin numbers mod small primes. For example, $p=5$ has treatments in \cite{deutschsagan} (Theorem 5.4) and \cite{CongruenceAutomaton} (Theorem 3) that we can procedurally extract from Corollary \ref{Mdiv} by writing down the first $3$ Motzkin numbers ($1, 1, 2$) and the first $5$ central trinomial coefficients ($1,1,3,7,19\equiv 1,1,3,2,4$) and checking finitely many congruences to find all forms of $(n)_p$ for which $M_n\equiv 0$.\\

Note that we require $T\ind_n\not\equiv 0\pmod p$ for every $n<p$ since this implies the same statement holds for all $n$ by Proposition \ref{lucas}. Thus, we do not need to consider when $T\ind_q\equiv 0$ in setting the right-hand side of Proposition \ref{M3case} congruent to $0$. If $T\ind_n\equiv 0$ for some $n$ (below $p$), then Proposition 4 of \cite{LastPaper} tells us that $0$ has density $1$ in $M\ind_n$ (i.e., the set of indices $\{i\in\N\mid M\ind_i\equiv 0\pmod p\}$ has density one). Hence, including this assumption does not hinder us in our goal.\\

We now perform the bookkeeping of computing the density of 0. We begin by stating a result in the generality of Lemma \ref{sumform} so as to encapsulate the ideas of this method independently of the particulars discovered in Proposition \ref{M3case} by application of Theorem \ref{Tsym}.

\begin{lemma}\label{GenDensity}
If $a_n = \ct{P^n}$ where $P$ is symmetric and $\deg P = 1$, $b_n = \sum_{i=0}^h\alpha_ia_{n+i}$, $C_n = \min(h, p-n-1)$, and $p > h$, then if $a_n\equiv 0\pmod p$ for some $n$ ($n$ may be taken less than $p$), then asymptotic density of $0$ in $b_n\bmod p$ is $1$, and otherwise it is equal to
\begin{align*}
&\frac{\left\vert\left\{0\leq n<p-h\mid b_n\equiv 0\right\}\right\vert}{p}\\
&+ \frac{\left\vert\left\{p-h\leq n<p,0\leq m < p-1\mid a_m\sum_{i=0}^{C_n}\alpha_ia_{n+i}\equiv -a_{m+1}\sum_{i=C_n+1}^h\alpha_ia_{i - (p-n)}\right\}\right\vert}{(p-1)(p+1)}\\
&+ \frac{\left\vert\left\{p-h\leq n<p,0\leq m < p-1\mid a_m\sum_{i=0}^{C_n}\alpha_ia_{n+i}\equiv -a_{p-1}a_{m+1}\sum_{i=C_n+1}^h\alpha_ia_{i - (p-n)}\right\}\right\vert}{(p-1)p(p+1)}
\end{align*}
\end{lemma}
\begin{proof}
The case where there exists $n$ such that $a_n\equiv 0\pmod p$ is treated in Proposition 4 of \cite{LastPaper}.\\

By Lemma \ref{sumform}, $b_n\equiv 0\pmod p$ if and only if either $n_0 < p-h$ and $b_{n_0}\equiv 0$ or else if $n_0\geq p-h$ and $a_ma_{p-1}^k\sum_{i=0}^{C_{n_0}}\alpha_ia_{n_0+i} \equiv -a_{m+1}\sum_{i=p-n_0}^h\alpha_ia_{i-(p-n_0)}$. Note that by an application of Theorem \ref{Tsym} (mentioned beneath its proof), we have that $a_{p-1}^2\equiv 1$ and so $a_{p-1}^k$ is congruent to $a_{p-1}$ or $1$ depending on the parity of $k$. In the set of strings, $(n)_p = qm(p-1)^kn_0$, over $\F_p$, the distribution on $k$ is geometric (with success probability of a trial being $\frac{1}{p}$) so that it is even $\frac{p}{p+1}$ of the time and odd $\frac{1}{p+1}$ of the time.\\

Putting this all together, each element of the set $\left\{0\leq n<p-h\mid b_n\equiv 0\right\}$ contributes $\frac{1}{p}$ to the density of $0$ in $b_n\bmod p$ since $\frac{1}{p}$ of the strings over $\F_p$ end with such $n$. Likewise each element of the set $\left\{p-h\leq n<p,0\leq m < p-1\mid a_m\sum_{i=0}^{C_n}\alpha_ia_{n+i}\equiv -a_{m+1}\sum_{i=C_n+1}^h\alpha_ia_{i - (p-n)}\right\}$ corresponds to $k$ being even and thus contributes $\frac{1}{p}\cdot\frac{1}{p-1}\cdot\frac{p}{p+1} = \frac{1}{(p-1)(p+1)}$ to the density of $0$, while $\left\{p-h\leq n<p,0\leq m < p-1\mid a_m\sum_{i=0}^{C_n}\alpha_ia_{n+i}\equiv -a_{p-1}a_{m+1}\sum_{i=C_n+1}^h\alpha_ia_{i - (p-n)}\right\}$ corresponds to $k$ being odd and thus contributes $\frac{1}{p}\cdot\frac{1}{p-1}\cdot\frac{1}{p+1} = \frac{1}{(p-1)p(p+1)}$ to the density.
\end{proof}

We now apply Lemma \ref{GenDensity} to the generalized Motzkin numbers using the equations from Corollary \ref{Mdiv} in place of those from Lemma \ref{sumform}. Furthermore, reasoning about the densities in $T\ind_n\bmod p$ further yields a description of the density of all other values in $M\ind_n\bmod p$ in terms of the density of $0$, $D_0$.

\begin{prop}\label{MainResult}
If $T\ind_n\equiv 0\pmod p$ for some $n<p$, then the asymptotic density of $0$ in $M\ind_n$ is $1$. Otherwise, for all $p>2$, it is equal to
\begin{align*}
D_0 = &\frac{\left\vert\left\{n < p-2\mid M\ind_n\equiv 0\right\}\right\vert}{p}\\
&+ \frac{2\left\vert\left\{m<p-1\mid bT\ind_m\equiv T\ind_{p-1}T\ind_{m+1}\right\}\right\vert}{(p-1)(p+1)}\\
&+ \frac{2\left\vert\left\{m<p-1\mid bT\ind_m\equiv T\ind_{m+1}\right\}\right\vert}{(p-1)p(p+1)}.
\end{align*}

Furthermore, if $\left\{T\ind_n\mid n\in\F_p\right\}$ generates $\F_p^\times$ as a multiplicative group, then the asymptotic density of any fixed $k>0$ in $M\ind_n\bmod p$ is $\frac{1 - D_0}{p-1}$.
\end{prop}
\begin{proof}
$D_0$ is given by combining the lemma with Corollary \ref{Mdiv} (we may assume $p\nmid b^2-4a^2$, as otherwise $T\ind_{p-1}\equiv 0$) noting that the factors of $2$ come from the fact that $n$ takes two values when $p-2\leq n < p$ (namely $p-2$ and $p-1$) and the corollary shows that in both cases we count the same set of elements.\\

To prove the last claim, we model $T\ind_n$ (as a function from $\N$ to $\F_p^\times$) by a directed multi-graph whose states are the elements of $\F_p^\times$, and whose $p$ outgoing transitions are labeled by $k\in\F_p$ where a transition from $i$ labeled $k$ goes to state $j = i\cdot T\ind_k\bmod p$. By Proposition \ref{lucas}, we can read the value of $T\ind_n$ off of this graph by starting at the state $1$ and following the transitions corresponding to the digits, $(n)_p$; the state we end at is the value of $T\ind_n\bmod p$. As it happens, each state also has exactly $p$ incoming transitions, since the number of transitions from $i$ to $j$ is equal to the number of transitions from $1$ to $i^{-1}j$ so that counting the number of these transitions from all $i$ (to a fixed $j$) is equal to counting all transitions leaving $1$, of which there are $p$. $\left\{T\ind_n\mid n\in\F_p\right\}$ generating $\F_p^\times$ as a multiplicative group is equivalent to the statement that this directed graph is connected. Treating this graph as a Markov process (where every transition has probability $\frac{1}{p}$), our assumption implies that the process is irreducible. Each state having a self-loop (labeled by $0$) makes the process aperiodic and irreducible so that every state converges over time to the stable state. All in-degrees and out-degrees of our graph being $p$ implies that the stable state is the uniform state $\begin{pmatrix}
\frac{1}{p-1} & \cdots & \frac{1}{p-1}
\end{pmatrix}$. This implies that the density of each element of $\F_p^\times$ in the sequence $T\ind_n$ is $\frac{1}{p-1}$.\\

Finally, from Proposition \ref{M3case} we can then see that each value of $\F_p^\times$ must have equal density in $M\ind_n\bmod p$ (because of the factor of $T\ind_q$ in each case).
\end{proof}

There is also an alternative direct argument for the final statement of the proposition that doesn't appeal to Proposition \ref{M3case}. We only sketch the argument here to avoid introducing unnecessary background. The argument goes that with density $1$, a random walk on the Rowland-Zeilberger automaton (see \cite{rowlandzeilberger}) of $M\ind_n\bmod p$ reaches either the $0$ state or else the subgraph corresponding to $T\ind_n\bmod p$ (which is exactly the graph described two paragraphs ago) and then cannot escape in either case. In fact, one interpretation of this paper is as describing what proportion of random walks on this automaton reach the $0$ state. For example, under this interpretation, the first summand in $D_0$ above corresponds to the probability that walks go to the $0$ state on their very first step.

\begin{conj}\label{mult_conj}
For all $a, b\in \Z$, if $p$ does not divide any member of $S_p = \left\{T\ind_n\mid n\in\F_p\right\}$, then $S_p$ generates $\F_p^\times$ as a multiplicative group. Equivalently, $\left\{T\ind_n\bmod p\mid n\in\N\right\} = \F_p^\times$. Consequently, all values of $\F_p^\times$ appear in $T\ind_n, M\ind_n\bmod p$, and with \textbf{equal} density.
\end{conj}

In order to complete the full picture, note that when $p=2$, we get that $$D_0 = \left\lbrace\begin{array}{cc}1 & \text{if }b\equiv 0\\0 & \text{if }a\equiv 0\text{ and }b\equiv 1\\\frac{1}{3} & \text{if }a\equiv 1\text{ and }b\equiv 1\end{array}\right..$$ When $b = T\ind_1\equiv 0$ this is because the argument that $D_0=1$ when $p\mid T\ind_n$ for some $n$ (including $n=1$) does not require $p>2$. Otherwise, when $a\equiv 0$ we have $M\ind_n\bmod 2 = \ct{1^n(1-x^2)} = 1$ for all $n$, and when $a\equiv 1$ we have $M\ind_n\bmod 2\equiv M_n$ and it is not hard to manually derive that the density of $0$ in $M_n\bmod 2$ is $\frac{1}{3}$ (for example, treat Figure 1 of \cite{CongruenceAutomaton} as a Markov process and compute the probability of ending at each sink).\\

Finally, when $a=b=1$ above, we get a simple formula for the density of Motzkin numbers divisible by $p$. Furthermore, applying Theorem \ref{Tsym} to this simple formula yields the lower bound of $D_0\geq\frac{2}{p(p-1)}$, which has appeared in the conclusion of \cite{burnspaper}.
\begin{cor}\label{MainCorollary}
The asymptotic density of $0$ in $M_n\bmod p$ is $1$ if $p\mid T_n$ for some $n<p$, and otherwise, for $p>2$, it is
\begin{align*}
D_0 = &\frac{\left\vert\left\{n<p-2\mid M_n\equiv 0\right\}\right\vert}{p}\\
&+ \frac{2\left\vert\left\{m<p-1\mid T_m\equiv T_{p-1}T_{m+1}\right\}\right\vert}{(p-1)(p+1)}\\
&+ \frac{2\left\vert\left\{m<p-1\mid T_m\equiv T_{m+1}\right\}\right\vert}{(p-1)p(p+1)}.
\end{align*}

Additionally, if Conjecture \ref{mult_conj} is true for $a=b=1$, then all other elements of $\F_p$ have density $\frac{1-D_0}{p-1}$. Lastly, $D_0\geq \frac{2}{p(p-1)}$ is a lower bound.
\end{cor}
\begin{proof}
Since $T_0 = T_1 = 1, T_2=3, T_{p-3}\equiv (-3)^{\frac{p-1}{2}-2}T_2 = -(-3)^{\frac{p-1}{2}-1}\equiv -T_{p-2}$, and $T_{p-1}\equiv (-3)^{\frac{p-1}{2}}\equiv\pm 1$, we are guaranteed that each of the second and third sets in $D_0$ above are non-empty (since $T_0\equiv T_1$ and $T_{p-3}\equiv (-1)T_{p-2}$) giving us a lower bound of $$D_0\geq \frac{2\cdot 1}{(p-1)(p+1)} + \frac{2\cdot 1}{(p-1)p(p+1)} = \frac{2}{p(p-1)}.$$
\end{proof}

\begin{subsection}{Application to Other Sequences}\label{other_seq}
To further demonstrate the fruitfulness of this approach, we quickly derive analogous results for a few additional sequences.

\begin{prop}
Let $s_n = \ct{(x^{-1} + 1 + x)^nx}$, which is the sequence A005717 of \cite{oeis}. Then the asymptotic density of $0$ in $s_n\bmod p$ is $1$ if $p\mid T_n$ for some $n<p$ and otherwise, for $p>2$, it is $$D_0 = \frac{\left\vert\left\{m<p-1\mid T_m\equiv T_{p-1}T_{m+1}\right\}\right\vert + p\left\vert\left\{m<p-1\mid T_m\equiv T_{m+1}\right\}\right\vert}{(p-1)(p+1)}.$$ Additionally, $D_0\geq\frac{1}{p-1}$ is a lower bound.
\end{prop}
\begin{proof}
First note that $2s_n = T_{n+1} - T_n$. Consequently, Lemma \ref{GenDensity} tells us that
\begin{align*}
D_0 = &\frac{\left\vert\left\{0\leq n<p-1\mid s_n\equiv 0\right\}\right\vert}{p}\\
&+ \frac{\left\vert\left\{0\leq m < p-1\mid T_m\equiv T_{p-1}T_{m+1}\right\}\right\vert}{(p-1)(p+1)}\\
&+ \frac{\left\vert\left\{0\leq m < p-1\mid T_m\equiv T_{m+1}\right\}\right\vert}{(p-1)p(p+1)}
\end{align*}
and so our formula for $D_0$ follows from the observation that $s_n\equiv 0\Leftrightarrow T_n\equiv T_{n+1}$. Lastly, the lower bound follows from the fact that $T_0 = T_1 = 1$, that $T_{p-1}\equiv (-3)^{\frac{p-1}{2}}\equiv \pm 1$, and that, also by Theorem \ref{Tsym}, $$T_{p-3}\equiv (-3)^{\frac{p-1}{2}-2}\cdot 3 = -(-3)^{\frac{p-1}{2} - 1}\equiv -T_{p-2}\pmod p.$$ Therefore, $D_0\geq \frac{1 + p\cdot 1}{(p-1)(p+1)} = \frac{1}{p-1}$.
\end{proof}
Because $D_0 = \frac{1}{p-1}$ for $p=5,11,13$ and others, this bound is tight.\\

We turn now to the Riordan numbers, $R_n = \ct{(x^{-1} + 1 + x)^n(1-x)}$, which is the sequence A005043 of \cite{oeis}.
\begin{prop}\label{oneminusx}
The asymptotic density of $0$ in $R_n\bmod p$ is $1$ if $p\mid T_n$ for some $n<p$ and otherwise, for $p>2$, it is $$D_0 = \frac{\left\vert\left\{m<p-1\mid 3T_m\equiv T_{p-1}T_{m+1}\right\}\right\vert + p\left\vert\left\{m<p-1\mid 3T_m\equiv T_{m+1}\right\}\right\vert}{(p-1)(p+1)}.$$ Additionally, $D_0\geq\frac{1}{p-1}$ is a lower bound.
\end{prop}
\begin{proof}
First note that $R_n = -T_{n+1} + 3T_n$. Consequently, Lemma \ref{GenDensity} tells us that
\begin{align*}
D_0 = &\frac{\left\vert\left\{0\leq n<p-1\mid R_n\equiv 0\right\}\right\vert}{p}\\
&+ \frac{\left\vert\left\{0\leq m < p-1\mid 3T_m\equiv T_{p-1}T_{m+1}\right\}\right\vert}{(p-1)(p+1)}\\
&+ \frac{\left\vert\left\{0\leq m < p-1\mid 3T_m\equiv T_{m+1}\right\}\right\vert}{(p-1)p(p+1)}
\end{align*}
and so our formula for $D_0$ follows from the observation that $R_n\equiv 0\Leftrightarrow 3T_n\equiv T_{n+1}$. Lastly, the lower bound follows from the fact that $3T_1 = T_2 = 1$ and that, by Theorem \ref{Tsym}, $$-T_{p-1}\equiv -(-3)^{\frac{p-1}{2}}\cdot T_0 = 3(-3)^{\frac{p-1}{2} - 1}\equiv 3T_{p-2}\pmod p.$$ Therefore, $D_0\geq \frac{1 + p\cdot 1}{(p-1)(p+1)} = \frac{1}{p-1}$.
\end{proof}
Because $D_0 = \frac{1}{p-1}$ for $p=5,11,23,31$ and others, this bound is tight.\\

We get the same result for A005773 of \cite{oeis}, which is no coincidence.
\begin{prop}\label{oneplusx}
Let $s_n = \ct{(x^{-1} + 1 + x)^n(1+x)}$, which is the sequence A005773 of \cite{oeis}. Then the asymptotic density of $0$ in $s_n\bmod p$ is $1$ if $p\mid T_n$ for some $n<p$ and otherwise, for $p>2$, it is $$D_0 = \frac{\left\vert\left\{m<p-1\mid T_m\equiv -T_{p-1}T_{m+1}\right\}\right\vert + p\left\vert\left\{m<p-1\mid T_m\equiv -T_{m+1}\right\}\right\vert}{(p-1)(p+1)}.$$ Additionally, $D_0\geq\frac{1}{p-1}$ is a lower bound.
\end{prop}
\begin{proof}
First note that $2s_n = T_{n+1} + T_n$. Consequently, Lemma \ref{GenDensity} tells us that
\begin{align*}
D_0 = &\frac{\left\vert\left\{0\leq n<p-1\mid s_n\equiv 0\right\}\right\vert}{p}\\
&+ \frac{\left\vert\left\{0\leq m < p-1\mid T_m\equiv -T_{p-1}T_{m+1}\right\}\right\vert}{(p-1)(p+1)}\\
&+ \frac{\left\vert\left\{0\leq m < p-1\mid T_m\equiv -T_{m+1}\right\}\right\vert}{(p-1)p(p+1)}
\end{align*}
and so our formula for $D_0$ follows from the observation that $s_n\equiv 0\Leftrightarrow T_n\equiv -T_{n+1}$. Lastly, the lower bound follows from the fact that $T_0 = T_1 = 1$ and that, by Theorem \ref{Tsym}, $$T_{p-3}\equiv (-3)^{\frac{p-1}{2}-2}\cdot 3 = -(-3)^{\frac{p-1}{2} - 1}\equiv -T_{p-2}\pmod p.$$ Therefore, $D_0\geq \frac{1 + p\cdot 1}{(p-1)(p+1)} = \frac{1}{p-1}$.
\end{proof}
Because $D_0 = \frac{1}{p-1}$ for $p=5,11,23,31$ and others, this bound is tight.\\

In fact, these formulas show an unexpected connection between these last two sequences.
\begin{cor}\label{same_density} Let $s_n = \ct{(x^{-1} + 1 + x)^n(1+x)}$ and $R_n = \ct{(x^{-1} + 1 + x)^n(1-x)}$, which are A005773 and A005043 of \cite{oeis} respectively. Then for every prime, $p$, the densities of $0$ in the sequences $(s_n\bmod p)_{n\in\N}$ and $(R_n\bmod p)_{n\in\N}$ are equal.
\end{cor}
\begin{proof}
This result follows from Propositions \ref{oneminusx} and \ref{oneplusx} for these densities along with a simple application of Theorem \ref{Tsym}: Since $3T_m\equiv 3(-3)^{m-\frac{p-1}{2}}T_{p-1-m} = -(-3)^{m+1-\frac{p-1}{2}}T_{p-1-m}$ and $T_{m+1}\equiv (-3)^{m+1-\frac{p-1}{2}}T_{p-2-m}\pmod p$, we have that $3T_m\equiv T_{m+1}\Leftrightarrow -T_{p-1-m}\equiv T_{p-2-m}$. Therefore, $\left\vert\left\{m<p-1\mid 3T_m\equiv T_{m+1}\right\}\right\vert = \left\vert\left\{m<p-1\mid T_m\equiv -T_{m+1}\right\}\right\vert$ because these sets biject via $m\mapsto p-2-m$. Likewise, Theorem \ref{Tsym} tells us that $3T_m\equiv -(-3)^{m+1-\frac{p-1}{2}}T_{p-1-m}\equiv -T_{p-1}(-3)^{m+1}T_{p-1-m}$ and $T_{p-1}T_{m+1}\equiv (-3)^{\frac{p-1}{2}}(-3)^{m+1-\frac{p-1}{2}}T_{p-2-m} = (-3)^{m+1}T_{p-2-m}$ and so $3T_m\equiv T_{p-1}T_{m+1}\Leftrightarrow -T_{p-1}T_{p-1-m}\equiv T_{p-2-m}$. Therefore, $\left\vert\left\{m<p-1\mid 3T_m\equiv T_{p-1}T_{m+1}\right\}\right\vert = \left\vert\left\{m<p-1\mid T_m\equiv -T_{p-1}T_{m+1}\right\}\right\vert$ because these sets biject via $m\mapsto p-2-m$.
\end{proof}
\end{subsection}

\bibliographystyle{plain}
\bibliography{DensityAndSymmetryInTheMotzkinNumbers}{}

\begin{thebibliography}{1}

\bibitem{burnspaper}
Rob Burns.
\newblock Structure and asymptotics for {M}otzkin numbers modulo primes using
  automata.
\newblock 2017.
\newblock Preprint arXiv:1703.00826.

\bibitem{deutschsagan}
Emeric Deutsch and Bruce~E. Sagan.
\newblock Congruences for {C}atalan and {M}otzkin numbers and related
  sequences.
\newblock {\em Journal of Number Theory}, 117(1):191--215, 2006.

\bibitem{motzkinsurvey}
Robert Donaghey and Louis~W Shapiro.
\newblock Motzkin numbers.
\newblock {\em Journal of Combinatorial Theory, Series A}, 23(3):291--301,
  1977.

\bibitem{LastPaper}
Nadav Kohen.
\newblock Uniform recurrence in the {M}otzkin numbers and related sequences mod
  $p$.
\newblock 2024.
\newblock Preprint arXiv:2403.00149.

\bibitem{oeis}
{OEIS Foundation Inc. (2024)}.
\newblock The {O}n-{L}ine {E}ncyclopedia of {I}nteger {S}equences.
\newblock Published electronically at \texttt{https://oeis.org}.

\bibitem{CongruenceAutomaton}
Narad Rampersad and Jeffrey Shallit.
\newblock Congruence properties of combinatorial sequences via {W}alnut and the
  {R}owland-{Y}assawi-{Z}eilberger automaton.
\newblock {\em Electronic Journal of Combinatorics}, 29(3):Paper No. 3.36, 13,
  2022.

\bibitem{rowlandzeilberger}
Eric~S. Rowland and Doron Zeilberger.
\newblock A case study in meta-automation: automatic generation of congruence
  automata for combinatorial sequences.
\newblock {\em Journal of Difference Equations and Applications}, 20:973--988,
  2013.

\bibitem{WeightedMotzkin}
Wen-jin Woan.
\newblock A recursive relation for weighted {M}otzkin sequences.
\newblock {\em Journal of Integer Sequences}, 8(1):Article 05.1.6, 2005.

\end{thebibliography}

\end{document}